\documentclass[11pt,reqno]{amsart}
\usepackage{amsmath,amssymb,rawfonts,enumerate}
\usepackage{tikz}
\usepackage{color}
\usepackage{soul}
\usepackage[english]{babel}
\usepackage{graphicx}
\usepackage[utf8]{inputenc}
\newcommand{\be}{\begin{equation}}
\newcommand{\ee}{\end{equation}}

\newtheorem{theorem}{Theorem}
\newtheorem{lemma}[theorem]{Lemma}
\newtheorem{defi}[theorem]{Definition}
\newtheorem{prop}[theorem]{Proposition}
\newtheorem{coro}[theorem]{Corollary}

\long\def\alert#1{\parindent2em\smallskip\hbox to\hsize%
{\hskip\parindent\vrule%
\vbox{\advance\hsize-2\parindent\hrule\smallskip\parindent.4\parindent%
\narrower\noindent#1\smallskip\hrule}\vrule\hfill}\smallskip\parindent0pt}

\begin{document}  
\title{A Note on Locally Compact Subsemigroups of Compact Groups}
\author{Julio C\'esar Hern\'andez Arzusa and Karl H. Hofmann}
\maketitle 

\begin{abstract}
An elementary proof is given for the fact
 that every locally compact subsemigroup of a compact topological group 
is a closed subgroup. A sample consequence is that  every 
commutative cancellative 
pseudocompact locally compact Hausdorff topological semigroup  with open 
shifts is a compact topological group.

\noindent
{\it Mathematics Subject Classification 2010:} Primary: 20M10, 22A25, \hfill\break
 22C05;  Secondary: 54B30, 54H10.

\end{abstract}

\bigskip

\noindent  {\bf The Basics}\qquad A  nonempty subset $S$ of a group satisfying  $SS\subseteq S$ is called a subsemigroup. 
All topological spaces are considered Hausdorff spaces.

\begin{defi} \label{1.2} 
\begin{enumerate}[a)]
    \item A topological group is called 
{\em Weil-adapted}  if the closure of each subsemigroup is a group.
\item  For a subset $A$ of a topological space $X$, a point $x\in A$ is a 
\emph{conditionally inner point} of $A$, if $x$ has an open 
neighborhood $U$ such that $x\in \overline{A}\cap U\subseteq A$. 
\end{enumerate}


\end{defi}

The following lemma has a remarkably elementary, self-contained,  and straightforward proof:\newpage

\begin{lemma}\label{297} Let $S$ be a subsemigroup of a
Weil-adapted topological group $G$.  If $S$ has a conditionally inner point,
then $S$ is a closed subgroup of $G$.
\end{lemma}

\begin{proof} Define  $C=\overline S$.  Then the continuity of the multiplication implies that 
$CC=\overline{S}\,  \overline{S}\subseteq\overline{SS}\subseteq\overline{S}=C$.
 So $C$  is a closed subsemigroup of $G$ and therefore is a group since $G$ is Weil-adpated.
 It is no 
loss of generality to assume that $C=G$. So we assume now that $S$ is dense in $G$
and we must show $S=G$. 
Let $T$ be the interior of $S$, then $T\ne\emptyset$ since $S$ has a conditionally inner point. 
Take $s\in S$ and $t\in T$, since left translations of $G$ are homeomorphisms, 
$sT$ is an open neighborhood of $st$ and $sT\subseteq sS\subseteq S$.
Hence $st$ is contained in the interior $T$ of $S$, so $T$ is a left ideal of $S$.
Let $D=\overline T$, then $D$ is a subgroup of $G$ since $G$ is Weil-adapted. 
Now $SD=S\overline T\subseteq \overline{ST}\subseteq \overline T=D$. Therefore $S\subseteq D$,
and since $S$ is dense in $G$, we have $D=G$, that is, $T$ is dense in $G$. 
  Now the mapping  $x\mapsto x^{-1}$ from $G$ to $G$
is a homeomorphism, 
thus $T^{-1}$ is the dense interior of $S^{-1}$.
Let $H=T\cap T^{-1}$, then $H$ is open and  dense in $G$, too, and
in addition, is closed under both multiplication and inversion. 
Hence  $H$ is an open dense subgroup of  $G$. 
But any open subgroup of a topological group is closed (as complement of the union of
all other cosets) and so  $H=G$ follows. 
Then $G=H=T\cap T^{-1}\subseteq T\subseteq S$ finally shows $S=G$,  completing
 the proof of the lemma.                          $\hfill\square$
\end{proof}

\begin{lemma}\label{310} In a  topological group $G$ the following conditions are equivalent:
\begin{enumerate}[i)]
    \item $G$ is Weil-adapted.
    \item Each closed  subsemigroup of $G$ is a group.
    \item  For each $g\in G$, the closure $\overline{\{g,g^2,g^3,\dots\}}$ is a subgroup of $G$.
\end{enumerate}
\end{lemma}

\begin{proof} Trivially  $i)$ implies $ii)$, and since the closure of a subsemigroup of a
topological group is a subsemigroup according to the first step of the proof of Lemma 1,
also $ii)$ implies $iii)$. So we have to  
 prove that $iii)$ implies $i)$.  Indeed,  let $S$ be a subsemigroup of
 $G$ and let $s\in \overline S$ and consider $C=\overline{\{s,s^2,s^3,\dots\}}$.
Then $C$ is a subgroup of $G$ by $iii)$ and is contained in 
$\overline S$ by the definition of $C$.  So $s^{-1}\in C \subseteq \overline S$ which shows that $\overline S$ is a subgroup, which we had to show.$\hfill\square$
\end{proof}

\begin{lemma}\label{julio} Any subgroup of a Weil-adapted topological group is  Weil-adapated.
\end{lemma}

\begin{proof} Let $A$ be a subgroup of a Weil-adapted topological group $G$ and
let $S$ be a subsemigroup of $A$. Since $G$ is Weil-adapted, the closure
$\overline S$ of $S$ in $G$ is a group by Definition \ref{1.2}(a). Then
the closure $\overline S \cap A$ of $S$ in $A$ is a group as well. $\hfill\square$
\end{proof}

We recall Weil's Lemma saying that for an element $g$ in a locally compact 
group the subgroup $\{\dots,g^{-2},g^{-1},1,g,g^2,\dots\}$ is either 
isomorphic to the discrete group $\mathbb{Z}$, or else $\{g,g^2,g^3,\dots\}$ is dense
in a compact subgroup (see e.g.\  \cite{Hof}, 7.43). This explains the 
terminology of  Definition 1(a). Weil's Lemma also holds in any  
pro-Lie group by \cite{Mor}, 5.3.  
Accordingly, a locally compact group or a pro-Lie group is Weil-adapted if and only
if it does not contain  infinite discrete cyclic subgroups.
In particular, every compact group is Weil-adapted.

\smallskip

Now a precompact group $P$  has a compact completion. The latter is Weil-adapted, 
hence by Lemma \ref{julio}, $P$  is Weil-adapated. So Lemma \ref{297}
 implies the following corollary.

\begin{coro} \label{julius} A subsemigroup of a precompact group is a group if
it has conditionally inner points.
\end{coro}

In particular, any open subsemigroup of a precompact group is a group.
For example,  $\mathbb{Z}$ in its p-adic topology is precompact, hence is
Weil-adapted but $\mathbb{N}$ fails to be closed.  

\smallskip
We say that a subspace of a topological space is  \emph{conditionally open} 
if each of its points is a conditionally inner point.
Now recall  that a locally compact space  
is conditionally open in any Hausdorff space that 
contains it; the elementary  proof is an exercise and is 
provided  in \cite{engel}, 3.3.9.

\medskip

Accordingly,  the following conclusions are  immediate:

\begin{prop} \label{good}  Any locally compact subsemigroup  
of a Weil-adapted topo\-lo\-gi\-cal group
is a closed subgroup.
\end{prop}

\begin{coro} \label{cesar} Any locally compact subsemigroup of a
precompact group is a group.
\end{coro}

Note that any locally compact subgroup of a Hausdorff topological group is closed
(see e.g.\ \cite{Hof}, Corollary A4.24). As a consequence we have the following corollary.

\begin{coro} [{\sc F. Wright}, \cite{fwright}] \label{wright} 
Any locally compact subsemigroup of 
a compact group is a compact subgroup.
\end{coro}

\medskip

This concludes the essentially selfcontained part of this note.
The following discussion makes references to other publications.

\bigskip
\noindent {\bf Some Consequences} \qquad The literature exhibits a variety of sufficient conditions for 
a cancellative topological semigroup $S$ to be a topological group, Indeed this is true if
$S$ is
\begin{enumerate}[i)]
\item  compact (\cite[Proposition A4.34]{Hof}),
\item countably compact first countable (\cite[Corollary 5]{cc})
\item sequentially compact (\cite[Theorem 6]{can}),
\item commutative and  locally compact connected with open shifts (\cite[Theorem 4]{axioms}), or
\item commutative feebly compact first countable regular with open shifts 
(\cite[Theorem 3]{axioms}). 
\end{enumerate}

\noindent
Here a space is called \textit{feebly compact} 
if each locally finite open family is finite. What we call a shift
in a semigroup is frequently also called a translation.
Recall also that a space
is called \emph{pseudocompact} if every every real valued function on it is bounded.
The list can now be expanded if we first quote 
Corollary 2 of \cite{axioms} as follows:

\begin{prop} \label{2971} A   commutative cancellative 
locally compact pseudocompact topological semigroup with open shifts 
can be embedded in a compact topological group as a  dense open
subsemigroup.
\end{prop}

Now from Lemma \ref{297} and Proposition \ref{2971} we obtain the following corollary.

\begin{coro} \label{2972} Each commutative cancellative locally 
compact pseudocompact topological semigroup with open shifts 
is a compact topological group.
\end{coro}

It is well known that every locally compact pseudocompact 
topological group is a compact topological group 
(see \cite[Theorem 2.3.2]{pseudo}).  Corollary  \ref{2972} now  
confirms  this conclusion for a class of topological semigroups.

It may be helpful to recall the example of the circle group 
$\mathbb{T}=\mathbb{R}/\mathbb{Z}$ in which the subsemigroup 
$S=(\mathbb{Z}+\sqrt{2}\mathbb{N})/\mathbb{Z}$ is not a subgroup.
  
From Proposition  \ref{good} we know that a genuine subsemigroup of 
a compact group (or indeed any Weil-complete group) cannot be locally compact. 
For the additive group $\mathbb{R}$ of reals, for any
 positive real number
$r$ and for  {\it any}
open subset $S$ such that $]2r,\infty[\subseteq S\subseteq ]r,\infty[$,  
the subset $S$ is an open subsemigroup of $\mathbb{R}$.

\bigskip
\bigskip

\noindent Julio C\'esar Hernández Arzusa\\
Programa de Matem\'aticas\\
Universidad de Cartagena\\
{Campus San Pablo - Zaragocilla}\\
{ 130014, Cartagena, Colombia}\\
jhernandeza2@unicartagena.edu.co

\bigskip
\bigskip

\noindent 
Karl Heinrich Hofmann\\
Fachbereich Mathematik\\
Technische Universit\"at Darmstadt\\
Schlossgartenstra{\ss}e 7\\
64289 Darmstadt, Germany\\
hofmann@mathematik.tu-darmstadt.de

\end{document}